\documentclass{siamart1116}

\usepackage{amssymb}
\usepackage{amsfonts}
\usepackage{csquotes}
\usepackage{pgfplots} \pgfplotsset{compat=1.11}
\usepackage{microtype}

\newcommand{\pd}{\partial}
\newcommand{\ddx}[2]{\frac{{\mathrm d}{#1}}{{\mathrm d}{#2}}}

\newcommand{\pdx}[2]{\frac{\partial {#1}}{\partial {#2}}}
\newcommand{\idx}[4]{\int_{#1}^{#2} {#3}\,{\mathrm d}{#4}}

\newcommand{\vv}[1]{\mathbf{#1}}        % Vectors
\newcommand{\gv}[1]{\boldsymbol{#1}}    % Vectors  using Greek symbols
\newcommand{\sm}[1]{\mathsf{#1}}        % 'Algebraic' matrix.
\newcommand{\sv}[1]{\mathsf{#1}}        % 'Algebraic' vector.

\DeclareMathOperator{\card}{card}
\DeclareMathOperator{\supp}{supp}
\DeclareMathOperator{\clos}{clos}
\DeclareMathOperator{\inte}{int}
\DeclareMathOperator{\diam}{diam}
\DeclareMathOperator{\Span}{span}
\DeclareMathOperator{\meas}{meas}
\DeclareMathOperator{\cond}{cond}
\DeclareMathOperator{\diag}{diag}

\newcommand{\phiref}{\hat{\varphi}}
\newcommand{\bigO}[1]{\mathcal{O}({#1})}
\newcommand{\project}{\mathcal{P}}
\newcommand{\extend}{\mathcal{E}}
\newcommand{\solve}{\mathcal{S}}

\newsiamremark{remark}{Remark}
\newsiamremark{assumption}{Assumption}

\hypersetup{
  pdftitle={A Smooth Partition of Unity Finite Element Method for Vortex Particle Regularization},
  pdfauthor={Matthias Kirchhart and Shinnosuke Obi}
}

\title{A Smooth Partition of Unity Finite Element Method for Vortex Particle Regularization}
\author{%
 Matthias Kirchhart%
 \thanks{Department of Mechanical Engineering, Keio University, Japan. 
         (\email{kirchhart@keio.jp}, \email{obsn@mech.keio.ac.jp})}
 \and
 Shinnosuke Obi\footnotemark[1]}

\newcommand{\TheTitle}{A Smooth PUFEM for Particle Regularization} 
\newcommand{\TheAuthors}{M. Kirchhart and S. Obi}
\headers{\TheTitle}{\TheAuthors}

\begin{document}
\maketitle
% REQUIRED
%\begin{abstract}
%A finite element based formulation for smoothing particle approximations is
%presented. Globally smooth shape functions are constructed on a Cartesian grid by
%using a partition of unity approach. To deal with general geometries we use a 
%fictitious domain formulation and suggest a higher order stabilization term to
%ensure stability and convergence of the method. Numerical experiments confirm the
%predictions of the analysis and suggest that the complexity of the velocity
%computation can be improved from $\bigO{h^{-d}}$ to $\bigO{h^{-d/2}}$, where $h$
%denotes the particle spacing and $d$ the number of spatial dimensions.
%\end{abstract}
\begin{abstract}
We present a new class of $C^\infty$-smooth finite element spaces on Cartesian
grids, based on a partition of unity approach. We use these spaces to construct
smooth approximations of particle fields, i.\,e., finite sums of weighted Dirac
deltas. In order to use the spaces on general domains, we propose a fictitious
domain formulation, together with a new high-order accurate stabilization. Stability,
convergence, and conservation properties of the scheme are established. Numerical
experiments confirm the analysis and show that the Cartesian grid-size $\sigma$
should be taken proportional to the square-root of the particle spacing~$h$,
resulting in significant speed-ups in vortex methods.
\end{abstract}

% REQUIRED

% REQUIRED

\section{Introduction}
Vortex particle methods are numerical schemes for solving the incompressible
Navier--Stokes equations. Instead of their formulation in the primitive variables
velocity $\vv{u}$ and pressure, they make use of the equivalent formulation in
terms of the vorticity $\gv{\omega} = \nabla\times\vv{u}$. The core idea is most
easily explained in the two-dimensional, inviscid case, where the equation for
the scalar vorticity $\omega$ in a bounded domain~$\Omega$ then reads:
\begin{equation}\label{eqn:vte}
\pdx{\omega}{t} + (\vv{u}\cdot\nabla)\omega = 0\qquad \mbox{in $\Omega$.}
\end{equation}
Let us for the moment assume the velocity field would be known, in which case 
we have a linear equation. We then can discretize the initial vorticity field
with \emph{particles}:
\begin{equation}\label{eqn:pufem_particles}
\omega(t=0)\approx\omega_h(0) = \sum_{i=1}^{N}\Gamma_i\delta(x-x_i(0)),
\end{equation} 
where $\Gamma_i$ and $x_i$ denote the circulation and position of particle $i$,
and $\delta$ is the Dirac delta function. Such particle fields can be seen
as quadrature rules for integrating smooth functions $\varphi$ against the
vorticity $\omega$ we are aiming to approximate. Let us for example assume
we are given a quasi-uniform, shape-regular triangulation of the domain $\Omega$
of mesh-width~$h$. If one then applies a quadrature rule of exactness degree~$m$
with positive weights to each cell, one obtains a set of quadrature nodes $x_i$
with associated weights $w_i$. A particle field approximation could then
be obtained by setting $\Gamma_i:= w_i\omega(x_i)$ in \cref{eqn:pufem_particles}.
For such a particle field one can prove error-bounds of the form:
\begin{equation}
\Vert\omega-\omega_h\Vert_{W^{-(m+1),2}(\Omega)}\leq Ch^{m+1}\Vert\omega\Vert_{W^{m+1,2}(\Omega)},
\end{equation}
for $m+1>d/2$, where $d$ is the number of spatial dimensions. Here and throughout
this text the symbol $C$ refers to a generic positive constant which is
independent of the functions involved.

The reason for choosing this particular discretization is the availability of an
analytic solution of \cref{eqn:vte} in this case. If one modifies the particles'
positions according to $\ddx{x_i}{t} = \vv{u}(x_i(t))$,  $i = 1,\ldots,N$, the
resulting approximation $\omega_h$ fulfills the vorticity equation~\cref{eqn:vte}
\emph{exactly}, i.\,e., the only error in the approximation comes from the 
initialization error~\cite[Appendix~A]{cottet2000}.

In practice, however, the velocity field is of course not known and needs to be
retrieved from the vorticity. Let us for simplicity assume that the velocity
would vanish at the boundaries. In this case the Helmholtz decomposition theorem
tells us that the velocity can be retrieved through the Biot--Savart law without
any boundary integral terms:
\begin{equation}
\vv{u} = \vv{K}\star\omega,\ \vv{K}(x) = \frac{1}{2\pi}\frac{(-x_2,x_1)^\top}{|x|^2},
\end{equation}
where $\star$ denotes convolution. We have the following classical estimate due
to Calder\'on and Zygmund~\cite{calderon1952}:
\begin{equation}\label{eqn:calzyg}
\Vert\vv{K}\star\omega\Vert_{W^{1,2}(\Omega)}\leq C\Vert\omega\Vert_{L^2(\Omega)}.
\end{equation}
The problem is that the particle approximation $\omega_h\not\in L^2(\Omega)$ is
not smooth enough to apply this estimate; applying the Biot--Savart law to the
particle field directly yields a singular velocity field. The question we try to
answer in this paper is how to obtain an accurate, smooth approximation
$\omega_\sigma\in L^2(\Omega)$ from the particle field $\omega_h$, where $\sigma$
refers to a smoothing length, which will be defined precisely later. This problem
is called \emph{particle regularization.} Once such a smooth approximation has
been obtained, one closes the system of equations by setting $\vv{u}_\sigma :=
\vv{K}\star\omega_\sigma$ and modifying the particle's positions according to
$\ddx{x_i}{t} = \vv{u}_\sigma(x_i(t),t)$ instead. It is this natural treatment
of convection which makes vortex methods so appealing. Given an appropriate choice
of $\omega_\sigma$ one can show that the resulting method is essentially free of
artificial viscosity and conserves mass, circulation, linear momentum, and angular
momentum, and the energy of $\vv{u}_\sigma$ exactly~\cite[Section~2.6]{cottet2000}.
When extended to handle physical viscosity, this makes the method particularly
attractive for flows at high Reynolds numbers, see for example the recent work by
Yokota et al.~\cite{yokota2013}.

%\enlargethispage{1\baselineskip}
The most common approach to the regularization problem is to mollify the particle
field with a certain, radially symmetric \emph{blob-function} $\zeta_\sigma$:
$\omega_\sigma := \omega_h\star\zeta_\sigma$, where $\sigma$ denotes the radius
of the blob's core~\cite[Section~2.3]{cottet2000}. Many commonly used blob-functions
have infinite support, effectively extending $\omega_\sigma$ from $\Omega$ to
$\mathbb{R}^d$ and blurring the domain's boundaries. There are approaches to use
blob-functions with varying shapes near boundaries \cite{teng1982,lakkis2009} or
to use image particles outside of the domain~\cite[Section~4.5.2]{cottet2000}.
These approaches assume that the boundaries are flat and usually fail in the
presence of sharp corners or kinks. Another approach is to interpret the
particles' circulations as weighted function values $\Gamma_i = w_i\omega(x_i)$,
where the $w_i$ are weights from an underlying quadrature rule. While this is
strictly speaking only the case during the initialization stage, the approach is
then to create a triangulation of the domain using the particles' positions as
grid nodes and to use these values to construct a piece-wise linear
approximation~\cite{russo1994}. This requires a mesh to be regenerated at every
time-step, which is problematic as the particle field gets distorted over time.
In Vortex-in-Cell (VIC) schemes one uses interpolation formulas to obtain a
grid-based approximation of the vorticity field. In the vicinity of boundaries
these formulas need to be specifically adapted to the particular geometry at hand
and cannot be used for arbitrary domains~\cite{cottet2004}. In general the
regularization problem causes significant difficulties, and as Cottet and
Koumoutsakos point out in the introduction of their book~\cite{cottet2000}:
\enquote{To our knowledge there is no completely satisfactory solution for general
geometries, in particular because of the need to regularize vortices near the
boundary.}

In this work we try to address this problem with the help of a finite element
formulation. The non-smooth $W^{-(m+1),2}$-nature of the particle field forces
us to use shape functions that are globally $W^{m+1,2}$-smooth, which is not the
case for the classical, piecewise linear elements. The partition of unity finite
element method~(PUFEM) by Melenk and Babu\v{s}ka is a generalization of the
classical finite element method (FEM), which can be used to obtain such smooth
spaces. Even though already mentioned in their introductory paper~\cite{melenk1996},
there seems to have been little research in this direction. Duarte et al.~\cite{duarte2006}
describe an approach which only works for certain triangulations in two
dimensions.

The generation of globally smooth shape functions on general meshes in higher
dimensions is a well-known, hard problem. We instead consider simple Cartesian
grids, on which the construction of smooth shape functions is easier. We then
apply a fictitious domain approach to deal with general geometries. This
typically results in instabilities in the cut elements. Under the name \emph{ghost
penalty} Burman~\cite{burman2010} presented an effective and accurate
stabilization strategy for this problem, which has for example been successfully
applied to several other flow problems with cut elements~\cite{burman2012,%
gross2016,hansbo2013,kirchhart2016}. We use a similar, higher-order approach
inspired by Burman and Fern\'andez~\cite{burman2014} as well as Cattaneo et al.~%
\cite{cattaneo2015} to achieve accuracy and stability of the resulting
discretization.

The rest of this article is structured as follows. In \cref{sec:pufem} we
define and construct smooth PUFEM spaces and analyze some of their important
properties. In \cref{sec:formulation} we introduce a stabilized variational
formulation and prove its stability and convergence. The regularization
problem is then treated as a perturbation to this variational formulation.
Similar to the approach with blob-functions as mentioned above, the resulting
error can be split into regularization and quadrature error parts which
need to be carefully balanced. The analysis will show that the smoothing
parameter $\sigma$ should be taken proportional to the square-root of the
particle-spacing $h$ and that this choice is in a certain sense optimal. In
\cref{sec:experiments} we perform numerical experiments, confirming our
analysis. As a consequence of the quadratic relation between $h$ and $\sigma$,
the computation of the velocity field only has a computational complexity of
$\bigO{h^{-\frac{d}{2}}}$, enabling the use of particle numbers on desktop
workstations which were previously only possible on super computers. We
finish the article with concluding remarks and acknowledgements.

\section{Smooth Partition of Unity Finite Element Spaces}\label{sec:pufem}
\subsection{Basic Theory}
We begin this subsection by defining smooth partitions of unity, similar to
Melenk's and Babu\v{s}ka's theory~\cite{melenk1996}.
\begin{definition}[Smooth Partition of Unity]\label{def:pou}
Let $\Omega\subset\mathbb{R}^d$ be an open set and let $\lbrace\Omega_i\rbrace$
be an open cover of $\Omega$ satisfying a pointwise overlap condition:
\begin{equation}
\exists M\in\mathbb{N}:\ \forall x\in\Omega:\ 
\card\lbrace i\,|\,x\in\Omega_i\rbrace\leq M.
\end{equation}
Let $\lbrace\varphi_i\rbrace$ be a Lipschitz partition of unity subordinate to
the cover $\lbrace\Omega_i\rbrace$ satisfying
\begin{align}
\supp\varphi_i      &\subset \clos\Omega_i, \\
\sum_{i}\varphi_i(x)  &\equiv 1\ \text{on $\Omega$},\\
\vert\varphi_i\vert_{W^{k,\infty}(\mathbb{R}^d)} &\leq
C(k)(\diam{\Omega_i})^{-k}\quad k\in\mathbb{N}_0,\label{eqn:phinorm}
\end{align}
where the $C(k)$ are positive constants and the symbol
$\vert\cdot\vert_{W^{k,p}(\mathbb{R}^d)}$ refers to the Sobolev semi-norms:
\begin{equation}
\vert f\vert_{W^{k,p}(\mathbb{R}^d)} :=
\begin{cases}
\biggl(\sum_{|\alpha|=k}\Vert\pd^\alpha f\Vert_{L^p(\mathbb{R}^d)}^{p}\biggr)^{1/p} & p\in[1,\infty), \\
\max_{|\alpha|=k}\Vert\pd^\alpha f\Vert_{L^\infty(\mathbb{R}^d)} & p = \infty.
\end{cases}
\end{equation}
Then, $\lbrace\varphi_i\rbrace$ is called a smooth partition of unity subordinate
to the cover $\lbrace\Omega_i\rbrace$. The sets $\Omega_i$ are called patches.
\end{definition}
Using these functions $\lbrace\varphi_i\rbrace$, we can define the spaces for
the partition of unity finite element method (PUFEM).
\begin{definition}[PUFEM Spaces]\label{def:pufemspace}
Let $\lbrace\varphi_i\rbrace$ be a smooth partition of unity subordinate to
the open cover $\lbrace\Omega_i\rbrace$. For $P,k\in\mathbb{N}_0$, $p\in[1,\infty]$
we define polynomial enrichment spaces $V_i^P\subset W^{k,p}(\Omega\cap
\Omega_i):$
\begin{equation}
V_i^P := \Span\bigl\lbrace
x^\alpha\,\bigr|\,|\alpha|\leq P
\bigr\rbrace,
\end{equation}
and the PUFEM spaces $V_\sigma^P(\Omega)\subset W^{k,p}(\Omega):$
\begin{equation}
V_\sigma^P := \Span\bigl\lbrace
\varphi_iv_i\,\bigr|\,%
v_i\in V_i^P
\bigr\rbrace
\end{equation}
where  $\sigma := \max_i\diam{\Omega_i}$ refers to the maximum patch diameter.
\end{definition}

\begin{assumption}\label{ass:bramble-hilbert}
We will assume that the shapes of the domain $\Omega$ and the patches
$\lbrace\Omega_i\rbrace$ are such that we can apply the Bramble--Hilbert
lemma~\cite[Lemma (4.3.8)]{brenner2008}. In particular, we will assume that for
all $u\in W^{P+1,p}(\Omega\cap\Omega_i)$, $p\in[1,\infty]$, there exists a
$v_i\in V_i^P$ such that:
\begin{equation}\label{eqn:bramble-hilbert}
\vert u-v_i\vert_{W^{k,p}(\Omega\cap\Omega_i)}
\leq
C\sigma^{P+1-k}\vert u\vert_{W^{P+1,p}(\Omega\cap\Omega_i)}\ 
\forall k\in\mathbb{N}_0,\ k\leq P + 1,
\end{equation}
where the constant $C$ is independent of $\sigma$ and $u$.
\end{assumption}

We then have the following estimate, which is a straightforward generalization
of the result of Melenk and Babu\v{s}ka~\cite[Theorem 2.1]{melenk1996}.
\begin{theorem}\label{thm:approx}
Let $V_\sigma^P(\Omega)$ be as in \cref{def:pufemspace} and let~\cref{ass:bramble-hilbert}
be fulfilled. Then for any $u\in W^{P+1,p}(\Omega)$, $p\in[1,\infty]$, there exists
$\project u\in V_\sigma^P(\Omega)$ such that:
\begin{equation}
\vert u-\project u\vert_{W^{k,p}(\Omega)}\leq C\sigma^{P+1-k}\vert u\vert_{W^{P+1,p}(\Omega)}\ 
\forall k\in\mathbb{N}_0,\ k\leq P + 1,
\end{equation}
where the constant $C$ is independent of $\sigma$ and $u$.
\end{theorem}
\begin{proof}
Here, we will only consider the case $p\in[1,\infty)$; the proof for the case
$p=\infty$ is analogous.  With $v_i\in V_i^P$ as in \cref{ass:bramble-hilbert}
we set $\project u:=\sum_i\varphi_iv_i$. We may then write for any multi-index
$\alpha$ with $|\alpha|=k$:
\begin{equation}
\Vert\pd^\alpha(u-\project u)\Vert_{L^p(\Omega)}^p =
\bigl\Vert\pd^\alpha\sum_i(u-v_i)\varphi_i\bigr\Vert_{L^p(\Omega)}^p =
\bigl\Vert\sum_{\beta\leq\alpha}\sum_i{\binom{\alpha}{\beta}}
\pd^\beta\varphi_i\pd^{\alpha-\beta}(u-v_i)\bigr\Vert_{L^p(\Omega)}^p.
\end{equation}
Considering the absolute value of the expanded derivative on the right,
we obtain using H\"older's inequality:
\begin{equation}
\biggl|\sum_{\beta\leq\alpha}\sum_{i}\begin{pmatrix}\alpha\\\beta\end{pmatrix}%
\pd^\beta\varphi_i\pd^{\alpha-\beta}\bigl(u-v_i\bigr)\biggr|^p
\leq C(\alpha,p)\sum_{\beta\leq\alpha}\biggl|\sum_{i}%
\pd^\beta\varphi_i\pd^{\alpha-\beta}\bigl(u-v_i\bigr)\biggr|^p
\end{equation}
and thus:
\begin{equation}
\Vert\pd^\alpha(u-\project u)\Vert_{L^p(\Omega)}^p \leq C\sum_{\beta\leq\alpha}\ 
\bigl\Vert
\sum_{i}\pd^\beta\varphi_i\pd^{\alpha-\beta}(u-v_i)
\bigr\Vert_{L^p(\Omega)}^p.
\end{equation}
Now, using the fact that for every point $x\in\Omega$ there are at most $M$
non-zero terms in the sum over $i$, we obtain by again using H\"older's
inequality:
\begin{equation}
\biggl|\sum_{i}
\pd^\beta\varphi_i\pd^{\alpha-\beta}(u-v_i)
\biggr|^p
\leq C(M,p)
\sum_{i}\bigl|
\pd^\beta\varphi_i\pd^{\alpha-\beta}(u-v_i)
\bigr|^p
\end{equation}
After inserting this in the previous relation we may exchange the order of
summation. Using the fact that $\varphi_i \equiv 0$ outside $\Omega_i$
we obtain:
\begin{equation}
\Vert\pd^\alpha(u-\project u)\Vert_{L^p(\Omega)}^p \leq C\sum_{i}\sum_{\beta\leq\alpha}%
\Vert\pd^\beta\varphi_i\pd^{\alpha-\beta}(u-v_i)\Vert_{L^p(\Omega\cap\Omega_i)}^p.
\end{equation}
Applying \cref{eqn:phinorm}, \cref{eqn:bramble-hilbert}, and collecting all
the terms yields the claim.
\end{proof}

\subsection{Construction of a Smooth Partition of Unity}\label{sec:smoothpart}
In this subsection we are going to construct such a smooth partition of unity
using mollification. We will make use of the following two definitions.
\begin{definition}[Friedrichs' Mollifier]
The function:
\begin{equation}
\begin{split}
\zeta: \mathbb{R}\to[0,K^{-1}],\qquad x&\mapsto
\begin{cases}
0 & \mbox{if $|x|\geq \tfrac{1}{2}$,}  \\
K^{-1}\exp{\bigl(-\frac{1}{1-4x^2}\bigr)} & \mbox{else,}
\end{cases}\\
K &\approx 0.221\,996\,908\,084\,039\,719,
\end{split}
\end{equation}
is called Friedrichs' mollifier in one-dimensional space. The constant $K$ was
obtained numerically, such that $\Vert\zeta\Vert_{L^1(\mathbb{R})}=1$.
For spatial dimensions greater than $d=1$ we define Friedrichs' mollifier using
the product:
\begin{equation*}
\zeta: \mathbb{R}^d\to[0,K^{-d}],\qquad
(x_1,\ldots,x_d)\mapsto \prod_{i=1}^{d}\zeta(x_i),
\end{equation*} 
where under a slight abuse of notation, we reused the symbol $\zeta$.
\end{definition}
It is well known that $\zeta\in C_0^\infty(\mathbb{R}^d)$, and thus also
$\zeta\in W^{k,p}(\mathbb{R}^d),\ k\in\mathbb{N}_0,\ p\in[1,\infty]$. Furthermore,
we have $\supp{\zeta} = [-\tfrac{1}{2},\tfrac{1}{2}]^d$, which leads to the
following definition.
\begin{definition}[Cartesian Grid]\label{def:cartesiangrid}
Given $\sigma>0$, we define Cartesian grid points $x_i\in\mathbb{R}^d$,
$i\in\mathbb{Z}^d$, $x_i := (i_1\sigma,\ldots,i_d\sigma)$. With each grid point
we associate a patch~$\Omega_i$ and a patch-core~$\omega_i$: 
\begin{align}
\Omega_i &:= \bigl((i_1-1)\sigma,(i_1+1)\sigma\bigr) \times \ldots \times 
             \bigl((i_d-1)\sigma,(i_d+1)\sigma\bigr),\\
\omega_i &:= \bigl((i_1-\tfrac{1}{2})\sigma,(i_1+\tfrac{1}{2})\sigma\bigr) \times \ldots \times 
             \bigl((i_d-\tfrac{1}{2})\sigma,(i_d+\tfrac{1}{2})\sigma\bigr).
\end{align}
\end{definition}
\begin{figure}
\centering
\begin{tikzpicture}
\draw[step=1cm] (-3.5,-0.5) grid (3.5,2.5);
\filldraw[fill=black!35!white,thick] (1,1) rectangle node {$Q_j$} (2,2);
\filldraw[radius=0.1] (1,1) circle node [anchor=north east] {$x_j$};
\filldraw[fill=black!50!white,semitransparent,thick] (-3,0) rectangle (-1,2);
\filldraw[fill=black!75!white,semitransparent,thick] (-2.5,0.5) rectangle (-1.5,1.5);
\filldraw[radius=0.1] (-2,1) circle node [anchor=north east] {$x_i$};
\node [fill=white] at (-0.5,2) {$\omega_i$};
\node [fill=white] at (-0.5,0) {$\Omega_i$};
\draw (-1.75,1.25) -- (-0.75,1.95);
\draw (-1.25,0.25) -- (-0.75,0.05);
\end{tikzpicture}
\caption{\label{fig:cartesian_grid}An illustration of the Cartesian grid. On the
left a grid node~$x_i$ together with its associated patch-core~$\omega_i$ and
patch~$\Omega_i$. On the right another grid node $x_j$ with its associated element~$Q_j$.}
\end{figure}
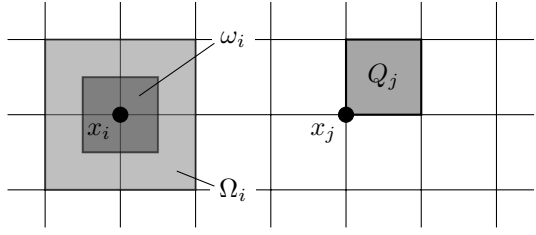
An illustration of these definitions is given in \cref{fig:cartesian_grid}.
It is obvious that the patches $\lbrace\Omega_i\rbrace$ form an open cover of
$\mathbb{R}^d$ with $M$ from \cref{def:pou} being equal to $2^d$. The patch-%
cores $\lbrace\omega_i\rbrace$ are pairwise disjoint and their closures form a
(non-open) cover of $\mathbb{R}^d$. Using these definitions, we are now ready to
construct smooth partition of unity functions $\lbrace\varphi_i\rbrace,
i\in\mathbb{Z}^d$.
\begin{lemma}
For a given $\sigma>0$ and $i\in\mathbb{Z}^d$ let $\varphi_i$ be the convolution
of the characteristic function $\chi_{\omega_i}$ of the patch-core $\omega_i$ with
the scaled Friedrichs' mollifier $\zeta_\sigma(x):=\sigma^{-d}\zeta(x/\sigma)$:
\begin{equation}
\varphi_i(x):= \bigl(\chi_{\omega_i}\star\zeta_\sigma\bigr)(x)
             = \idx{\omega_i}{}{\zeta_\sigma(x-y)}{y}.
\end{equation}
One then has:
\begin{align}
\varphi_i\in C_0^\infty(\mathbb{R}^d) &\mbox{ and } \supp{\varphi_i} = \clos{\Omega_i},\\
\sum_{i\in\mathbb{Z}^d}\varphi_i(x) &\equiv 1\quad x\in\mathbb{R}^d,\\
|\varphi_i|_{W^{k,p}(\mathbb{R}^d)} &\leq C(k)\sigma^{d/p-k}\quad k\in\mathbb{N}_0,\ p\in[1,\infty].
\end{align}
\end{lemma}
\begin{proof}
The first property directly follows from the classical properties of
mollification~\cite[sections 2.28 and 2.29]{adams2003}. For the second
property we immediately obtain:
\begin{equation}
\sum_{i\in\mathbb{Z}^d}\varphi_i(x) =
\sum_{i\in\mathbb{Z}^d}\idx{\omega_i}{}{\zeta_\sigma(x-y)}{y} =
\idx{\mathbb{R}^d}{}{\zeta_\sigma(x-y)}{y} = 1.
\end{equation}
For the last property we obtain with the help of Young's inequality for
convolutions for every multi-index $\alpha$ with $|\alpha|=k$:
\begin{equation}
\Vert\chi_{\omega_i}\star\pd^\alpha\zeta_\sigma\Vert_{L^p(\mathbb{R}^d)} \leq
\Vert\chi_{\omega_i}\Vert_{L^p(\mathbb{R}^d)} \Vert\pd^\alpha\zeta_\sigma\Vert_{L^1(\mathbb{R}^d)}  =
\sigma^{d/p-k}\Vert\pd^\alpha\zeta\Vert_{L^1(\mathbb{R}^d)}.
\end{equation}
\end{proof}
\begin{remark}\label{rmk:phiref}
There is no closed-form expression for the functions $\lbrace\varphi_i\rbrace$
available. However, it is important to notice that we have:
\begin{equation}
\varphi_i(x) \equiv\phiref\biggl(\frac{x-x_i}{\sigma}\biggr),\quad
\phiref(x)   :=\idx{(-\frac{1}{2},\frac{1}{2})^d}{}{\zeta(x-y)}{y}.
\end{equation}
Furthermore, $\phiref$ inherits the product structure of $\zeta$. In a computer
implementation it is thus sufficient to tabulate values for $\phiref$ corresponding
to the case $d=1$. We can then efficiently approximate $\phiref$ using, e.\,g.,
cubic Hermite splines. The graph of this function can be seen in \cref{fig:phi}. 
\end{remark}
\begin{figure}
\centering
\begin{tikzpicture}
\begin{axis}%
[
	xlabel = {$x$},
	ylabel = {$\phiref(x)$},
	grid   = {major},
	xmin   = {-1},
    xmax   = {+1},
    ymin   = {0},
	ymax   = {1},
]
\addplot [blue] table {phi.dat};
\end{axis}
\end{tikzpicture}
\caption{\label{fig:phi}An illustration of the one-dimensional partition of
unity function $\phiref$.}
\end{figure}
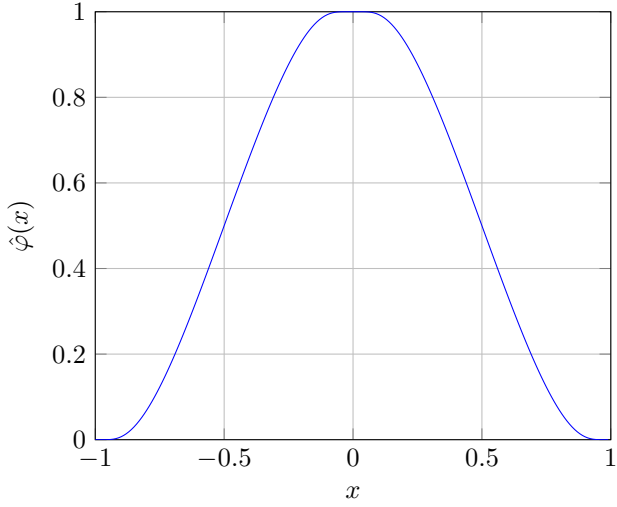

\subsection{Reference Element and Inverse Estimates}
In this subsection we illustrate that the smooth partition of unity constructed
in \cref{sec:smoothpart} leads to spaces that can be treated in a manner similar
to conventional finite element spaces. In particular, we can subdivide
$\mathbb{R}^d$ into \emph{elements}:
\begin{equation}\label{eqn:elements}
Q_i := (i_1\sigma,(i_1+1)\sigma)\times\ldots\times(i_d\sigma,(i_d+1)\sigma),
\qquad i\in\mathbb{Z}^d.
\end{equation}
Such an element is for example depicted on the right of \cref{fig:cartesian_grid}.
Every $Q_i$ may be seen as the image of the \emph{reference element}
$\hat{Q}:=(0,1)^d$ under the transformation $\Phi_i: \hat{Q}\to Q_i,\ 
\hat{x}\mapsto x_i + \sigma\hat{x}$. In every element $Q_i$ we have a fixed
set $\mathcal{J}_i$ of $2^d$ overlapping patches $\Omega_j$. Introducing:
\begin{equation}
\mathcal{B}_j^P := \biggl\lbrace
\underbrace{\biggl(\frac{x-x_j}{\sigma}\biggr)^\alpha}_{=:g_{j,\alpha}(x)}
\,\biggr|\,
|\alpha|\leq P
\biggr\rbrace,\qquad j\in\mathbb{Z}^d
\end{equation}
as bases for the enrichment spaces $V_j^P$, one quickly sees that within each
element $Q_i$ the basis functions $g_{j,\alpha}$ can be expressed in terms of
mapped reference functions $\hat{g}_{m,\alpha}$:
\begin{equation}
g_{j,\alpha}(x) = \bigl(\hat{g}_{m,\alpha}\circ\Phi_i^{-1}\bigr)(x),\qquad
x\in Q_i,\ j\in\mathcal{J}_i,
\end{equation}
where $m$ is the index of the node in the reference element that corresponds to
$x_j$. Due to \cref{rmk:phiref}, the same holds true for the partition of unity
functions $\varphi_j$. This allows us to infer the following  classical result,
which follows from a scaling argument and the norm-equivalence of finite-%
dimensional spaces~\cite[Lemma (4.5.3)]{brenner2008}.

\begin{lemma}[Inverse Estimates]\label{lem:investimates}
Let $V_\sigma^P(\Omega)$, $\sigma>0$, $P\in\mathbb{N}_0$ be as in
\cref{def:pufemspace}, with patches as in \cref{def:cartesiangrid}. Then, for
any element $Q_i\subset\Omega$ as defined in \cref{eqn:elements} that is completely
contained in the domain and every $v_\sigma\in V_\sigma^P(\Omega)$ one has:
\begin{equation}\label{eqn:investimates}
\Vert v_\sigma\Vert_{W^{l,p}(Q_i)}\leq C\sigma^{k-l}\Vert v_\sigma\Vert_{W^{k,p}(Q_i)},\quad
p\in[1,\infty],\ k,l\in\mathbb{N}_0,\ k\leq l,
\end{equation}
where the constant $C$ is independent of $\sigma$ and $i$.
\end{lemma}

\section{Stabilized Variational Formulation}\label{sec:formulation}
In this section we will introduce a stabilized variational formulation with the
aim of mimicking of the $L^2(\Omega)$-orthogonal projector onto $V_\sigma^P(\Omega)$.
As the inverse estimates~\cref{eqn:investimates} are not available for elements
$Q_i$ cut by the boundary $\pd\Omega$, we will employ a fictitious domain
approach. In order to ensure coercivity of the resulting bilinear form on the
entire fictitious domain, we will add a stabilization term in the cut cells. Once
consistency and stability of this formulation have been established, we will
model the regularization process as a perturbation to this variational problem.

\subsection{Basic Definitions and Conditions}
We will restrict ourselves to Hilbert spaces ($p=2$), due to the rich theoretical
framework available for this case. We will assume that the shape of the domain
$\Omega$ is such that we may apply the Stein extension
theorem~\cite[Section~5.24]{adams2003}, i.\,e., there exists a bounded linear
extension operator $\extend: W^{k,2}(\Omega)\to W^{k,2}(\mathbb{R}^d)$ for
any natural number $k$. We explicitly wish to include functions that do not
vanish on the boundary $\pd\Omega$. For this reason, for any domain
$\square\subset\mathbb{R}^d$, we will denote by $W^{-k,2}(\square):=
W^{k,2}(\square)^\prime$ the dual space of $W^{k,2}(\square)$. (Opposed to the
convention $W^{-k,2}(\square)=W_0^{k,2}(\square)^\prime$).

We will need certain geometrical definitions. To this end, let $\sigma>0$
be arbitrary but fixed. We define the fictitious domain $\Omega_\sigma$ as the
union of all elements that intersect the physical domain $\Omega$. Furthermore
we define cut and uncut elements $\Omega_\sigma^\Gamma$ and 
$\Omega_\sigma^\circ$, respectively:
\begin{equation}
\begin{split}
\Omega_\sigma        &:= \inte\ \bigcup\bigl\lbrace \clos{Q_i}\,\bigr|\,
\meas_d(Q_i\cap\Omega)>0\bigr\rbrace,\\
\Omega_\sigma^\Gamma &:= \inte\ \bigcup\bigl\lbrace \clos{Q_i}\,\bigr|\,
Q_i\in\Omega_\sigma\wedge Q_i\not\subset\Omega\bigr\rbrace,\\
\Omega_\sigma^\circ &:= \inte\ \bigcup\bigl\lbrace \clos{Q_i}\,\bigr|\,
Q_i\in\Omega_\sigma\wedge Q_i\subset\Omega\bigr\rbrace,
\end{split}
\end{equation}
with $Q_i$ as in \cref{eqn:elements}. Here, we write under a slight abuse of
notation $Q_i\in\Omega_\sigma^\Gamma$ if $Q_i\subset\Omega_\sigma^\Gamma$. These
domains obviously fulfill $\Omega_\sigma^\circ\subset\Omega\subset\Omega_\sigma$,
$\Omega_\sigma = \inte(\clos\Omega_\sigma^\circ\cup\clos\Omega_\sigma^\Gamma)$,
and $\Omega_\sigma^\circ\cap\Omega_\sigma^\Gamma=\emptyset$. Two elements $Q_i$ and
$Q_i'$ will be called neighbors if they share at least one node on the Cartesian
grid. We will make the following somewhat technical assumption: for every
$Q_i\in\Omega_\sigma^\Gamma$ there is a finite sequence of elements $(Q_i=Q_{i,1},
Q_{i,2},\ldots,Q_{i,K})\subset\Omega_\sigma^\Gamma$ with the following properties:
the number $K$ is bounded independent of $\sigma$, every pair of two subsequent
elements are neighbors, and $Q_{i,K}$ has a neighbor in $\Omega_\sigma^\circ$.
This condition means that one can always reach uncut elements from cut elements
in a bounded number of steps. For sufficiently fine Cartesian grids this 
condition is often fulfilled with $K=1$; if necessary it can be enforced by moving
additional elements from $\Omega_\sigma^\circ$ to $\Omega_\sigma^\Gamma$.

\subsection{Introduction of a Higher-order Stabilization Term}
The basic idea of the ghost penalty method is to control the norm of cut elements
by relating them to neighboring uncut elements. In the aforementioned articles~%
\cite{burman2012,hansbo2013,kirchhart2016}, for example, this is done by
controlling the norms of the gradient-jumps at element boundaries. However,
as our PUFEM spaces are globally smooth, they do not contain such jumps. 
Burman and Fern\'andez~\cite{burman2014} and Cattaneo et al.~\cite{cattaneo2015}
instead use the Brezzi--Pitk\"aranta stabilization~\cite{brezzi1984}. We will use
a higher-order variant of this idea and define the following bilinear form:
\begin{equation}
j(u_\sigma,v_\sigma) :=
\sigma^{2(P+1)}\sum_{Q_i\in \Omega_\sigma^\Gamma}\sum_{|\alpha|=P+1}
\idx{Q_i}{}{(\pd^\alpha u_\sigma)(\pd^\alpha v_\sigma)}{x},
\end{equation}
such that $j(u_\sigma,u_\sigma) = \sigma^{2(P+1)}|u_\sigma|_{W^{P+1,2}
(\Omega_\sigma^\Gamma)}^2$. We then obtain the following result:
\begin{lemma}
Let $u_\sigma\in V_\sigma^P(\Omega_\sigma)$. One then has with constants $c$
and $C$ independent of $\sigma$ and the position of $\pd\Omega$ relative to the
Cartesian grid:
\begin{equation}\label{eqn:jestimate}
c\Vert u_\sigma\Vert^2_{L^2(\Omega_\sigma)} \leq
\Vert u_\sigma\Vert^2_{L^2(\Omega_\sigma^\circ)} + j(u_\sigma,u_\sigma) \leq
C\Vert u_\sigma\Vert^2_{L^2(\Omega_\sigma)}.
\end{equation}
\end{lemma}
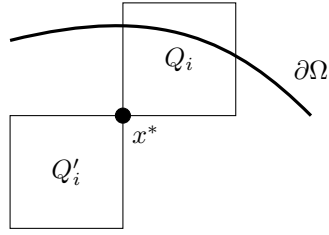
\begin{figure}
\centering
\begin{tikzpicture}
\draw (0,0) rectangle node {$Q_i'$}  (1.5,1.5);
\draw (1.5,1.5) rectangle node {$Q_i$} (3,3);
\filldraw (1.5,1.5) circle [radius=0.1] node [anchor=north west] {$x^*$};
\draw [very thick] (0,2.5) .. controls (2,3) and (3,2.5) .. (4,1.5) 
      node [above=10pt] {$\pd\Omega$};
\end{tikzpicture}
\caption{\label{fig:stablem}A cut element $Q_i\in\Omega_\sigma^\Gamma$
sharing a node $x^*$ with an uncut element~$Q_i'\in\Omega_\sigma^\circ$.}
\end{figure}
Before moving on to the proof of this lemma, let us remark that the stabilization
term is necessary. Look for example at the configuration shown in~\cref{fig:stablem}.
The partition of unity function corresponding to the node of $Q_i$ opposite to $x^*$
vanishes on $\Omega_\sigma^\circ$. Thus its $L^2(\Omega_\sigma)$-norm cannot be
controlled by looking at $\Omega_\sigma^\circ$ only, unless one adds a stabilization
term.
\begin{proof}
The second inequality directly follows from the inverse inequalities~%
\cref{eqn:investimates}. For the first inequality, let us first consider
the case $K=1$, i.\,e., a cut element $Q_i\in\Omega_\sigma^\Gamma$ and an
associated uncut element $Q_i'\in\Omega_\sigma^\circ$ which share a Cartesian
grid point $x^*$, as for example illustrated in \cref{fig:stablem}. This
configuration can be mapped to one of $P^{2^d}_2 = 4^d-2^d$ reference cases 
with reference elements $\hat{Q}$ and $\hat{Q}'$ using the transformation
$\hat{x}=\Phi^{-1}(x):=(x-x^*)/\sigma$, such that $\hat{x}^*=0$. For an
arbitrary function
$v_\sigma\in V_\sigma^P(Q_i\cup Q_i')$ one obtains with $\hat{v} :=
(v_\sigma\circ\Phi)\in V_1^P(\hat{Q}\cup\hat{Q}')$:
\begin{equation}
\Vert v_\sigma\Vert_{L^2(Q_i\cup Q_i')}^{2} = 
\sigma^d\Vert\hat{v}\Vert_{L^2(\hat{Q}\cup\hat{Q}')}^2.
\end{equation}
We claim that the following expression constitutes a norm on $V_1^P(\hat{Q}\cup\hat{Q}')$:
\begin{equation}
\Vert\hat{v}\Vert_*^2 := \Vert\hat{v}\Vert_{L^2(\hat{Q}')}^2 +
\vert\hat{v}\vert_{W^{P+1,2}(\hat{Q})}^2.
\end{equation}

It suffices to show that $\Vert\hat{v}\Vert_{*} = 0 \Longrightarrow\hat{v}= 0.$
From $\Vert\hat{v}\Vert_{L^2(\hat{Q}')}^2=0$ we obtain $\hat{v}\equiv 0$ on
$\hat{Q}'$ and due to the global smoothness of $\hat{v}$ also $\pd^\alpha\hat{v}
(\hat{x}^*)= 0$ for all multi-indices $\alpha\in\mathbb{N}_0^d$. From
$|\hat{v}|_{W^{P+1,2}(\hat{Q})}^2 = 0$ we obtain $\pd^\alpha\hat{v}\equiv 0,
|\alpha| = P + 1$ on $\hat{Q}$. Together with $\pd^\alpha\hat{v}(\hat{x}^*)=0$
this implies $\hat{v}\equiv 0$ on $\hat{Q}$ as well. Thus $\Vert\cdot\Vert_*$
is indeed a norm. After employing the norm-equivalence of finite-dimensional spaces,
we can transform back to $Q_i\cup Q_i'$ and obtain:
\begin{equation}
 \sigma^d\Vert\hat{v}\Vert_{L^2(\hat{Q}\cup\hat{Q}')}^2 \leq
C\sigma^d\Vert\hat{v}\Vert_*^2 = 
C\bigl(\Vert v_\sigma\Vert_{L^2(Q_i')}^{2} + \sigma^{2(P+1)}\vert v_\sigma\vert_{W^{P+1,2}(Q_i)}^{2}\bigr).
\end{equation}
The case $K>1$ with sequences of cells $(Q_i = Q_{i,1}, Q_{i,2},\ldots, Q_{i,K} = Q_i')$
follows by induction. Now, summing over all elements and using the finite overlap condition
$M=2^d$, the claim follows.
\end{proof}
Note that the proof crucially depends on the global smoothness of the spaces
$V_\sigma^P$. In particular, this stabilization does not work with the
conventional finite element spaces. As an example consider the case depicted
in \cref{fig:stablem}, and set $v_\sigma:=0$ on $Q_i'$, and $v_\sigma:=(x - x^*)\cdot y$
on $Q_i$, where $y\in\mathbb{R}^d$ is an arbitrary non-zero vector. 

\subsection{Stability and Convergence}
As described before, we are aiming to mimic the $L^2(\Omega)$-orthogonal
projector. To this end, we introduce the bilinear forms $a$ and $A$:
\begin{align}
a:&\ L^2(\Omega_\sigma)\times L^2(\Omega_\sigma)\to \mathbb{R},\quad
&(u,v)&\mapsto\idx{\Omega}{}{uv}{x},\\
A:&\ W^{P+1,2}(\Omega_\sigma)\times W^{P+1,2}(\Omega_\sigma)\to\mathbb{R},\quad
&(u_\sigma,v_\sigma)&\mapsto a(u_\sigma,v_\sigma) + \varepsilon j(u_\sigma,v_\sigma),
\label{eqn:bilinear_form_A}
\end{align}
where $\varepsilon>0$ denotes a user-defined stabilization parameter. We define
the variational problem as: given any $u$ for which the following makes sense,
find $u_\sigma\in V_\sigma^P(\Omega_\sigma)$ such that:
\begin{equation}\label{eqn:varprob}
A(u_\sigma,v_\sigma) = \idx{\Omega}{}{uv_\sigma}{x} \qquad
\forall v_\sigma\in V_\sigma^P(\Omega_\sigma).
\end{equation}
We then obtain the following two results.
\begin{theorem}[Stability]\label{thm:stability}
The bilinear form $A$ from \cref{eqn:bilinear_form_A} fulfills with a constant
$C(\varepsilon)$ independent of $\sigma$, $u_\sigma$, and the position of $\pd\Omega$
relative to the grid:
\begin{equation}
A(u_\sigma,u_\sigma)\geq C(\varepsilon)\Vert u_\sigma\Vert_{L^2(\Omega_\sigma)}^2\quad
\forall u_\sigma\in V_\sigma^P(\Omega_\sigma).
\end{equation}
\end{theorem}
\begin{proof}
For any $u_\sigma\in V_\sigma^P(\Omega_\sigma)$ one has with the help of
\cref{eqn:jestimate}:
\begin{multline}
A(u_\sigma,u_\sigma) =
\Vert u_\sigma\Vert_{L^2(\Omega)}^2 + \varepsilon j(u_\sigma,u_\sigma) \geq \\
\Vert u_\sigma\Vert_{L^2(\Omega_\sigma^\circ)}^2 + \varepsilon j(u_\sigma,u_\sigma)
\geq C(\varepsilon)\Vert u_\sigma\Vert_{L^2(\Omega_\sigma)}^2.
\end{multline} 
\end{proof}
\begin{theorem}[Convergence]\label{thm:convergence}
Let $u\in W^{P+1,2}(\Omega_\sigma)$. The solution $u_\sigma\in V_\sigma^P(\Omega_\sigma)$
of the variational problem \cref{eqn:varprob} then satisfies the following error
bound:
\begin{equation}
\Vert u_\sigma - u\Vert_{L^2(\Omega_\sigma)} \leq C(\varepsilon)
\sigma^{P+1}\Vert u\Vert_{W^{P+1,2}(\Omega_\sigma)},
\end{equation}
where the constant $C(\varepsilon)$ is independent of $\sigma$, $u$, and how the
boundary $\pd\Omega$ intersects the grid.
\end{theorem}
\begin{proof}
According to \cref{thm:approx}, there exists $\project u\in V_\sigma^P(\Omega_\sigma)$
such that:
\begin{equation}\label{eqn:approx2}
\vert\project u-u\vert_{W^{k,2}(\Omega_\sigma)}\leq C\sigma^{P+1-k}
\vert u\vert_{W^{P+1,2}(\Omega_\sigma)}\qquad k\in\mathbb{N}_0,\ k\leq P + 1.
\end{equation}
We may write:
\begin{equation}
\Vert u_\sigma - u\Vert_{L^2(\Omega_\sigma)} \leq
\Vert u_\sigma - \project u\Vert_{L^2(\Omega_\sigma)} +
\Vert \project u - u\Vert_{L^2(\Omega_\sigma)}.
\end{equation}
For the second term we can apply relation~\cref{eqn:approx2}. For the first term
we obtain with \cref{thm:stability} and the fact that $u_\sigma$ solves~\cref{eqn:varprob}:
\begin{multline}
\Vert \project u - u_\sigma\Vert_{L^2(\Omega_\sigma)}^2 \leq C(\varepsilon)
A(\project u-u_\sigma,\project u-u_\sigma) = \\
C(\varepsilon)\biggl(\bigl(\project u-u,\project u-u_\sigma\bigr)_{L^2(\Omega)} + 
\varepsilon j(\project u,\project u-u_\sigma)\biggr) \leq \\
C(\varepsilon)\biggl(
\Vert\project u-u\Vert_{L^2(\Omega_\sigma)}\Vert\project u-u_\sigma\Vert_{L^2(\Omega_\sigma)} +
\varepsilon j(\project u,\project u)^{1/2}j(\project u-u_\sigma,\project u-u_\sigma)^{1/2}
\biggr),
\end{multline}
where we used the Cauchy--Schwarz inequality in the last step. Noting that by the
inverse estimates~\cref{eqn:investimates} we have:
\begin{equation}
j(\project u-u_\sigma,\project u-u_\sigma)^{1/2} \leq
C\Vert\project u-u_\sigma\Vert_{L^2(\Omega_\sigma)}
\end{equation}
and together with \cref{eqn:approx2}:
\begin{equation}
j(\project u,\project u)^{1/2} \leq C\sigma^{P+1}
\Vert\project u\Vert_{W^{P+1,2}(\Omega_\sigma)} \leq 
C\sigma^{P+1}
\Vert u\Vert_{W^{P+1,2}(\Omega_\sigma)}.
\end{equation}
After dividing both sides by $\Vert \project u - u_\sigma\Vert_{L^2(\Omega_\sigma)}$
we thus obtain:
\begin{equation}\label{eqn:omfg}
\Vert\project u - u_\sigma\Vert_{L^2(\Omega_\sigma)} \leq
C(\varepsilon)\biggl(\Vert\project u-u\Vert_{L^2(\Omega_\sigma)} +
\varepsilon\sigma^{P+1}\Vert u\Vert_{W^{P+1,2}(\Omega_\sigma)}\biggr).
\end{equation}
Again applying \cref{eqn:approx2} to the first term yields the claim.
\end{proof}

\subsection{Influence of the Quadrature Error}
In vortex methods we are only given a particle field, i.\,e., a quadrature rule
for integrating smooth functions against the underlying vorticity we are aiming
to approximate. Furthermore the bilinear form $A$ can usually only be computed
approximately, using numerical quadrature. In this subsection we are analyzing
the influence of these additional sources of error.

We will assume that the bilinear form $j$ can be computed exactly. This is
justified as it is sufficient to perform computations on the reference element
$\hat{Q}$, which can be done up to arbitrary precision a priori. As for the
bilinear form $a$, we will assume the availability of quadrature rules $I_m$
satisfying error bounds of the following form:
\begin{equation}
\biggl|\idx{\Omega}{}{f}{x}-I_m(f)\biggr| \leq
Ch^{m+1}\vert f\vert_{W^{m+1,1}(\Omega)} \quad f\in W^{m+1,1}(\Omega).
\end{equation}
Such error estimates typically arise from the application of quadrature rules of
exactness degree~$m$ and positive weights to the cells of a quasi-uniform triangulation
of the domain $\Omega$ of mesh-width~$h$. Note that due to the global smoothness
of the PUFEM spaces, these quadrature rules \emph{do not need to be aligned} with
the Cartesian grid. We will write $a_h(u_\sigma,v_\sigma):= I_m(u_\sigma v_\sigma)$
for the resulting approximate bilinear form. For $u_\sigma,v_\sigma\in
V_\sigma^P(\Omega_\sigma)$ one then obtains with the help of H\"{o}lder's inequality
and the inverse estimates \cref{eqn:investimates}:
\begin{multline}\label{eqn:aerror}
|a(u_\sigma,v_\sigma)-a_h(u_\sigma,v_\sigma)|          \leq 
Ch^{m+1}|u_\sigma v_\sigma|_{W^{m+1,1}(\Omega_\sigma)} \leq \\
Ch^{m+1}\sum_{|\alpha|=m+1}\sum_{\beta\leq\alpha}
\begin{pmatrix}\alpha \\ \beta\end{pmatrix}
\Vert\pd^\beta u_\sigma \pd^{\alpha-\beta}
               v_\sigma\Vert_{L^1(\Omega_\sigma)} \leq \\
Ch^{m+1}\sum_{|\alpha|=m+1}\sum_{\beta\leq\alpha}
\begin{pmatrix}\alpha \\ \beta\end{pmatrix}
\Vert\pd^\beta u_\sigma\Vert_{L^2(\Omega_\sigma)}
\Vert\pd^{\alpha-\beta} v_\sigma\Vert_{L^2(\Omega_\sigma)} \leq \\
Ch^{m+1}\sigma^{-(m+1)}
\Vert u_\sigma\Vert_{L^2(\Omega_\sigma)}
\Vert v_\sigma\Vert_{L^2(\Omega_\sigma)}.
\end{multline}
This will require us to couple $h$ and $\sigma$ through a relation like
$h = \sigma^s$, for some $s>1$. We then obtain coercivity of
$A_h(u_\sigma,v_\sigma):= a_h(u_\sigma,v_\sigma) +
j(u_\sigma,v_\sigma)$:
\begin{multline}
A_h(u_\sigma,u_\sigma) = A(u_\sigma,u_\sigma)-
\bigl(a(u_\sigma,u_\sigma)-a_h(u_\sigma,u_\sigma)\bigr) \geq \\
\bigl( C(\varepsilon) - Ch^{m+1}\sigma^{-(m+1)} \bigr)
\Vert u_\sigma\Vert^2_{L^2(\Omega_\sigma)} \geq
C(\varepsilon)\Vert u_\sigma\Vert^2_{L^2(\Omega_\sigma)},
\end{multline}
where the last constant $C(\varepsilon)$ is independent of $\sigma$, $u$, and the
position of $\pd\Omega$ relative to the Cartesian grid, for $\sigma>0$ small
enough, $h=\sigma^s$, $s>1$.

For the particle field $u_h$ we will assume an error bound of the following form:
\begin{equation}\label{eqn:particleerror}
\Vert u_h - u \Vert_{W^{-(m+1),2}(\Omega)}\leq
Ch^{m+1}\Vert u\Vert_{W^{m+1,2}(\Omega)},
\end{equation}
which is the typical form arising in vortex methods~\cite{cottet2000}. Again,
the particle field does not in any way need to be aligned to the Cartesian grid.
Collecting all of the previous results, we are ready to prove the main result
of this article.
\begin{theorem}\label{thm:perturbedprob}
Let $h=\sigma^s$, $s>1$, and denote $k:=\max\lbrace P, m\rbrace$. Let
$u\in W^{k+1,2}(\Omega)$ and let the particle approximation $u_h\in W^{-(m+1),2}(\Omega)$
satisfy the error bound \cref{eqn:particleerror}. Then for $\sigma>0$ small enough the
solution $u_\sigma\in V_\sigma^P(\Omega_\sigma)$ of the perturbed variational problem:
\begin{equation}\label{eqn:perturbedprob}
A_h(u_\sigma,v_\sigma) = \langle u_h,v_\sigma\rangle\quad\forall v_\sigma\in V_\sigma^P(\Omega_\sigma)
\end{equation}
satisfies the following error bound:
\begin{equation}
\Vert u - u_\sigma\Vert_{L^2(\Omega)} \leq C(\varepsilon)
\bigl(
\sigma^{P+1} +
h^{m+1}\sigma^{-(m+1)}
\bigr)\Vert u\Vert_{W^{k+1,2}(\Omega)},
\end{equation}
where the constant $C(\varepsilon)$ is independent of $\sigma$, $u$, and the position
of $\pd\Omega$ relative to the Cartesian grid.
\end{theorem}
\begin{proof}
Let $v_\sigma\in V_\sigma^P(\Omega_\sigma)$ denote the solution of the unperturbed
variational problem~\cref{eqn:varprob}, with $u$ extended to $\extend u$ by the
Stein extension operator. With the help of the coercivity of $A_h$ one then obtains:
\begin{multline}
\Vert u_\sigma-v_\sigma\Vert_{L^2(\Omega_\sigma)}^2 \leq C(\varepsilon)
A_h(u_\sigma-v_\sigma,u_\sigma-v_\sigma) = \\
C(\varepsilon)\biggl(
\langle u_h-u,u_\sigma-v_\sigma\rangle + 
A(v_\sigma,u_\sigma-v_\sigma) - A_h(v_\sigma,u_\sigma-v_\sigma)
\biggr).
\end{multline}
Application of the error bounds \cref{eqn:particleerror} and \cref{eqn:aerror}
as well as the inverse estimates \cref{eqn:investimates} yields:
\begin{multline}
\Vert u_\sigma-v_\sigma\Vert_{L^2(\Omega_\sigma)}^2 \leq C(\varepsilon)
\bigl(
h^{m+1}\sigma^{-(m+1)}\Vert u\Vert_{W^{m+1,2}(\Omega)}
\Vert u_\sigma - v_\sigma\Vert_{L^2(\Omega_\sigma)} + \\
h^{m+1}\sigma^{-(m+1)}
\Vert v_\sigma\Vert_{L^2(\Omega_\sigma)}
\Vert u_\sigma - v_\sigma\Vert_{L^2(\Omega_\sigma)}
\bigr).
\end{multline}
and thus:
\begin{equation}
\Vert u_\sigma-v_\sigma\Vert_{L^2(\Omega_\sigma)} \leq
C(\varepsilon)h^{m+1}\sigma^{-(m+1)}
\bigl(
\Vert u\Vert_{W^{m+1,2}(\Omega)} +
\Vert v_\sigma\Vert_{L^2(\Omega_\sigma)}
\bigr).
\end{equation}
Nothing that
$\Vert v_\sigma\Vert_{L^2(\Omega_\sigma)}\leq 
C(\varepsilon)\Vert u\Vert_{W^{P+1,2}(\Omega)}$ we obtain:
\begin{equation}\label{eqn:lol}
\Vert u_\sigma-v_\sigma\Vert_{L^2(\Omega_\sigma)} \leq
C(\varepsilon)h^{m+1}\sigma^{-(m+1)}\Vert u\Vert_{W^{k+1,2}(\Omega)}. 
\end{equation}
Now, by the triangle inequality:
\begin{equation}
\Vert u_\sigma- u\Vert_{L^2(\Omega)} \leq
\Vert u_\sigma- v_\sigma\Vert_{L^2(\Omega)} +
\Vert v_\sigma-u\Vert_{L^2(\Omega)}
\end{equation}
the claim follows by applying~\cref{thm:convergence} and the boundedness of
$\extend$ to the second term.
\end{proof}
The part $\sigma^{P+1}\Vert u\Vert_{W^{k+1,2}(\Omega)}$ is called the smoothing
error; $\sigma$ roughly corresponds to the blob-width in conventional vortex
particle methods. The second part is called the quadrature error; choosing
$s=1+\tfrac{P+1}{m+1}$ balances both terms. In the next section we will
illustrate that the choice $P=m$  does not only \enquote{feel natural}, but
also yields optimal results in a certain sense. In this case we obtain with
$s=2$ an overall convergence rate of $\bigO{\sigma^{P+1}}=\bigO{h^{\frac{1}{2}(P+1)}}$.

\subsection{Optimality of the Smoothed Solution}
In this subsection we will assume that we can apply the bilinear form $A$ exactly,
i.\,e., $A_h = A$. Furthermore we assume $P=m$. We will show that the smoothed
solution $u_\sigma$ then satisfies the same asymptotic error bound as $u_h$. We
will need the following corollary of~\cref{thm:convergence}.
\begin{corollary}\label{thm:corollary}
The solution operator $\solve$ to the problem~\cref{eqn:varprob} is bounded:
\begin{equation}
\Vert\solve u\Vert_{W^{P+1,2}(\Omega_\sigma)}\leq C(\varepsilon)
\Vert       u\Vert_{W^{P+1,2}(\Omega_\sigma)}\qquad\forall u\in W^{P+1,2}(\Omega_\sigma).
\end{equation}
\end{corollary}
\begin{proof}
One has $\Vert\solve u\Vert_{W^{P+1,2}(\Omega_\sigma)}\leq
\Vert\solve u - u\Vert_{W^{P+1,2}(\Omega_\sigma)} + \Vert u \Vert_{W^{P+1,2}(\Omega_\sigma)}.$
Using $\project$ from \cref{thm:approx}, we furthermore obtain:
\begin{math}
\Vert\solve u - u\Vert_{W^{P+1,2}(\Omega_\sigma)} \leq
\Vert\solve u - \project u\Vert_{W^{P+1,2}(\Omega_\sigma)} +
\Vert\project u - u\Vert_{W^{P+1,2}(\Omega_\sigma)}.
\end{math}
The second term can be bounded by $C\Vert u \Vert_{W^{P+1,2}(\Omega_\sigma)}$ due
to the boundedness of $\project$. The first term can be bounded by first
applying the inverse estimates~\cref{eqn:investimates} followed by estimate~\cref{eqn:omfg}.
\end{proof}
\begin{theorem}[Optimality]\label{thm:optimality}
Let the conditions of \cref{thm:perturbedprob} be fulfilled. Furthermore assume
that $A_h = A$, $P = m$, and $s = 2$. Then the smoothed solution $u_\sigma$ fulfills:
\begin{equation}
\Vert u_\sigma - u\Vert_{W^{-(P+1),2}(\Omega)}\leq
C(\varepsilon)h^{P+1} \Vert u\Vert_{W^{P+1,2}(\Omega)}.
\end{equation}
\end{theorem}
\begin{proof}
Let $\varphi\in W^{P+1,2}(\Omega)$ be arbitrary but fixed. With~$\project$
from \cref{thm:approx} and the Stein extension operator~$\extend$ one has:
\begin{equation}
\idx{\Omega}{}{(u_\sigma-u)\varphi}{x} =
\idx{\Omega}{}{(u_\sigma-u)(\varphi-\project\extend\varphi)}{x} +
\idx{\Omega}{}{(u_\sigma-u)\project\extend\varphi}{x}.
\end{equation}
For the first term we obtain with the Cauchy--Schwarz inequality, \cref{thm:approx},
and~\cref{thm:perturbedprob}:
\begin{multline}
\idx{\Omega}{}{(u_\sigma-u)(\varphi-\project\extend\varphi)}{x} \leq
\Vert u - u_\sigma\Vert_{L^2(\Omega)}\Vert \varphi-\project\extend\varphi\Vert_{L^2(\Omega)} \\ \leq 
C(\varepsilon)\sigma^{2(P+1)}\Vert u\Vert_{W^{P+1,2}(\Omega)}
\Vert\varphi\Vert_{W^{P+1,2}(\Omega)}.
\end{multline}
For the second term one has:
\begin{multline}
\idx{\Omega}{}{(u_\sigma-u)\mathcal{PE}\varphi}{x} =
A(u_\sigma,\project\extend\varphi) - \idx{\Omega}{}{u\project\extend\varphi}{x}
-\varepsilon j(u_\sigma,\project\extend\varphi) = \\
\langle u_h-u,\project\extend\varphi\rangle - \varepsilon j(u_\sigma,\project\extend\varphi) \leq \\
C\biggl( \Vert u_h-u\Vert_{W^{-(P+1),2}(\Omega)}\Vert\varphi\Vert_{W^{P+1,2}(\Omega)} +
\varepsilon\sigma^{2(P+1)}\Vert u_\sigma\Vert_{W^{P+1,2}(\Omega_\sigma)}
\Vert\varphi\Vert_{W^{P+1,2}(\Omega)}\biggr).
\end{multline}
It remains to show that $\Vert u_\sigma\Vert_{W^{P+1,2}(\Omega_\sigma)} \leq
C(\varepsilon)\Vert u\Vert_{W^{P+1,2}(\Omega)}$. To see this, note that we have:
\begin{multline}
\Vert u_\sigma\Vert_{W^{P+1,2}(\Omega_\sigma)} \leq
\Vert\solve\extend u\Vert_{W^{P+1,2}(\Omega_\sigma)} +
\Vert u_\sigma - \solve\extend u\Vert_{W^{P+1,2}(\Omega_\sigma)} \leq\\
C(\varepsilon)\Vert u\Vert_{W^{P+1,2}(\Omega)} +
\Vert u_\sigma -  \solve\extend u\Vert_{W^{P+1,2}(\Omega_\sigma)}.
\end{multline}
Applying inequality~\cref{eqn:lol} to the second term, collecting all the terms,
and noting that $\sigma = \sqrt{h}$ yields the result.
\end{proof}

\subsection{Conservation Properties}
In the introduction we mentioned the conservation properties of vortex methods
as one of their highlights. In this section we make some brief remarks on some of
these properties under the assumption that $A_h = A$. For brevity, we will focus on the
two-dimensional case, but remark that all of the results we present here analogously
hold in three-dimensions.

The conserved quantities circulation, linear momentum, and angular momentum are
given by $I_0=\idx{\Omega}{}{1\cdot\omega}{x}$, $\vv{I}_1=\idx{\Omega}{}{(x_2,-x_1)^\top\omega}{x}$,
and $I_2=\idx{\Omega}{}{|x|^2\omega}{x}$, respectively~\cite[Section~1.7]{majda2001}.
Noting that the stabilization term $j$ vanishes if one of its arguments is a
polynomial of total degree less than $P$, one obtains for the solution $\omega_\sigma$
of~\cref{eqn:varprob} with right-hand side $\omega_h$:
$(\omega_\sigma,x^\alpha)_{L^2(\Omega)} = \omega_h(x^\alpha)$ for all $|\alpha|\leq P$.
For $P=1$ we consequently conserve $I_0$ and $\vv{I}_1$, for $P=2$ one additionally
conserves angular momentum $I_2$. This is important, because in vortex methods
body forces are often computed using the relation $\vv{F} = -\rho\ddx{\vv{I}_1}{t}$,
where $\rho$ denotes the fluid's density.

\section{Numerical Experiments}\label{sec:experiments}
We have now established the necessary results to return to our original motivation.
Given a particle approximation $\omega_h\in W^{-(m+1),2}(\Omega)$ of the vorticity $\omega$
that satisfies an error-bound of the form $\Vert\omega_h-\omega\Vert_{W^{-(m+1),2}(\Omega)}
\leq Ch^{m+1}\Vert\omega\Vert_{W^{m+1,2}(\Omega)}$, we want to obtain a smooth
approximation $\omega_\sigma$, such that we can compute the corresponding induced
velocity field using the Biot--Savart law $\vv{u}_\sigma =\vv{K}\star\omega_\sigma$.
One can then use this approximate velocity field to advance $\omega_h$ in time
by convecting the particles according to $\ddx{x_i}{t}(t)=\vv{u}_\sigma(x_i(t),t)$.

In \cref{sec:pufem} we introduced the spaces $V_\sigma^P(\Omega)$ that can be used
as test-spaces for the particle field $\omega_h$. In \cref{sec:formulation} we modeled
the regularization problem as a perturbation to a stabilized $L^2$-projection onto
the spaces $V_\sigma^P(\Omega)$. The analysis indicated that one should choose $P=m$
and $\sigma=\sqrt{h}$, resulting in an a-priori error estimate of
$\Vert\omega_\sigma-\omega\Vert_{L^2(\Omega)}\leq C\sigma^{P+1}\Vert\omega\Vert_{W^{P+1,2}(\Omega)}$.
The Calder\'on--Zygmund inequality~\cref{eqn:calzyg} then tells us that one may expect
$\Vert\vv{u}_\sigma-\vv{u}\Vert_{W^{1,2}(\Omega)}\leq C\sigma^{P+1}\Vert\omega\Vert_{W^{P+1,2}(\Omega)}$
for the resulting velocity field $\vv{u}_\sigma := \vv{K}\star\omega_\sigma$. This
analogously holds in the three-dimensional case.

In this section we perform several numerical experiments. We will first describe the
experimental setup. We then perform experiments on a scalar particle field and confirm
the results of our analysis. In particular, the experiments will show that the common
practice of choosing $\sigma$ proportional to $h$ instead of $\sqrt{h}$ does not lead
to convergent schemes. We will then illustrate the practicality of our scheme, by
approximating a vector-valued vorticity field, computing its induced velocity field,
and measuring the error. We finish this section with experiments on the condition
number of the resulting systems and its dependence of the stabilization parameter
$\varepsilon$.

\subsection{Setup}\label{sec:setup}
We define our computational test domain as $\Omega=(-\tfrac{1}{2},\tfrac{1}{2})^3$.
While this is one of the simplest cases for mesh-based methods, due to its sharp
corners and edges it is one of the hardest for conventional vortex blob methods. 
In order to obtain quadrature rules which are not aligned to the Cartesian grid,
the mesh generator Gmsh~\cite{geuzaine2009} was used to obtain a tetrahedral
mesh of the domain, consisting of 24 tetrahedra with maximum edge-length $h=1$. 
The quadrature rules are obtained by applying the mid-point rule to this mesh
and its subsequent uniform refinements from level $l=0$ down to level $l = 8$,
corresponding to $h = 2^{-8}\approx 0.004$ and $N = 402\,653\,184$ quadrature nodes.

Preliminary experiments showed good results for a stabilization parameter of
$\varepsilon = 0.001$. Unless explicitly stated otherwise, we will use this value
for all of our computations. We will use degree $P=1$ for the PUFEM spaces, set
$\sigma := Ch^{1/s}$, and experiment on various choices of $C$ and~$s=1,2$. For
the integration of the bilinear form $A_h$ we use the following approach: if in
a pair of basis functions one of them has cut support, we use the same quadrature
rule as for the particle field. Otherwise precomputed values from the reference
element $\hat{Q}$ are used. The resulting systems of equations are solved using
the conjugate gradient method, where we apply a simple diagonal scaling as
preconditioner. The iteration was stopped when a relative residual of $10^{-12}$
was reached. This was usually the case after less than 100 iterations, with some
exceptions for coarse refinement levels $l$ and the case $C=0.5, s = 1$.

\subsection{Scalar Particle Field}\label{sec:scalar_field}
The common practice to choose the smoothing length $\sigma$ proportional to $h$
may in special cases be justified with the analysis of Cottet and Koumoutsakos~%
\cite[Section 2.6]{cottet2000}. They assume that the quadrature rules used are
of infinite order, essentially corresponding to the case $m=\infty$. Such rules,
however, typically only exist in very special cases, such as a cube with periodic
or zero boundary conditions. To show that this approach does not work in a
more general setting, we aim to approximate the following function:
\begin{equation}
u(x) = \cos{(4\pi x_1)}\qquad
x\in (-\tfrac{1}{2},\tfrac{1}{2})^3.
\end{equation}
This function \emph{does not vanish at the boundary.} The application of
conventional blob-methods would thus blur the boundaries and lead to only slowly
converging schemes. We define the particle field as $u_h:=\sum_{i=1}^N w_i u(x_i)
\delta(x-x_i)$, with $\delta$ denoting the Dirac Delta, and $x_i$ and $w_i$ being
the positions and weights of the mid-point quadrature rule applied to the
tetrahedra of the mesh at various refinement levels~$l$.

\Cref{fig:error_s1} shows the error $\Vert u - u_\sigma\Vert_{L^2(\Omega)}$ for
$\sigma = Ch$ for various choices of $C$ at different refinement levels. Choosing
$C=0.5$ results in approximations with large errors, which do not decrease
significantly under mesh refinement. The case $l=8$ was not computed due to the
large memory requirements. The other curves exhibit similar behavior: in the
beginning and intermediate stages the error decreases, however, only at an
approximately linear, not quadratic rate. This rate further decreases and
approaches zero under mesh refinement, confirming the predicted bound of the
quadrature error $\bigO{h^{m+1}\sigma^{-(m+1)}}=\bigO{1}$. Choosing
larger values $C$ somewhat delays but does not prevent this effect, at the cost of
larger errors on coarse refinement levels.

\Cref{fig:error_s2} shows the corresponding error for the case $\sigma=C\sqrt{h}$.
All choices of $C$ lead to convergent schemes which approach the predicted
convergence rate of $\bigO{h}$. In our experiments, smaller choices of $C$
lead to smaller errors; however choosing $C$ too small causes larger errors
in the coarser cases. In our test case a choice somewhere between $C=0.25$ and
$C=0.5$ seems to be optimal.

\begin{figure}
\centering
	\begin{tikzpicture}
	\begin{semilogyaxis}%
	[
		title  = {$s=1$},
		xlabel = {Refinement Level~$l$},
		ylabel = {$L^2$-Error},
		grid   = {major},
		xmin   = {0},
	    xmax   = {8},
	    ymin   = {1e-3},
		ymax   = {1},
	    legend entries = {$C=2$, $C=1.5$, $C=1$, $C=0.5$},
	    legend cell align = {left},
	]
	\pgfplotsset{xtick={0,1,2,3,4,5,6,7,8}}
	\pgfplotsset{every axis legend/.append style={
	at={(.05,.05)},
	anchor=south west}}
	\addplot table[x=l,y=C_2_00]{cosine_s1.dat};
	\addplot table[x=l,y=C_1_50]{cosine_s1.dat};
	\addplot table[x=l,y=C_1_00]{cosine_s1.dat};
	\addplot table[x=l,y=C_0_50]{cosine_s1.dat};
	\addplot[domain=3:5]{2^(-(x+2.5))} node[below=5,pos=0.5] {$\bigO{h}$};
	\end{semilogyaxis}
	\end{tikzpicture}
\caption{\label{fig:error_s1}$L^2$-Error of the smoothed approximation
$u_\sigma$ in the case $\sigma = Ch$.}
\vspace{\baselineskip}
\centering
	\begin{tikzpicture}
	\begin{semilogyaxis}%
	[
		title  = {$s=2$},
		xlabel = {Refinement Level~$l$},
		ylabel = {$L^2$-Error},
		grid   = {major},
		xmin   = {0},
		xmax   = {8},
	    ymin   = {1e-3},
		ymax   = {1},
	    legend entries = {$C=1.5$, $C=1.0$, $C=0.5$, $C=0.25$},
	    legend cell align = {left},
	]
	\pgfplotsset{xtick={0,1,2,3,4,5,6,7,8}}
	\pgfplotsset{every axis legend/.append style={
	at={(.05,.05)},
	anchor=south west}}
	\addplot table[x=l,y=C_1_50]{cosine_s2.dat};
	\addplot table[x=l,y=C_1_00]{cosine_s2.dat};
	\addplot table[x=l,y=C_0_50]{cosine_s2.dat};
	\addplot table[x=l,y=C_0_25]{cosine_s2.dat};
	\addplot[domain=6:8]{2^(-(x+1.5))} node[above=5,pos=0.66] {$\bigO{h}$};
	\end{semilogyaxis}
	\end{tikzpicture}
\caption{\label{fig:error_s2}$L^2$-Error of the smoothed approximation
$u_\sigma$ in the case $\sigma = C\sqrt{h}$.}
\end{figure}

\subsection{Vector-valued Particle Field and Velocity Evaluation}
In this section we show that our scheme can drastically reduce the cost of the
computationally most expensive part of vortex methods, the velocity evaluation.
To this end, we prescribe:
\begin{equation}
\vv{u}(x) :=
\begin{pmatrix}
x_2 \\ -x_1 \\ 0
\end{pmatrix}
\exp\biggl(-\frac{1}{1-4|x|^2}\biggr)\qquad
x\in (-\tfrac{1}{2},\tfrac{1}{2})^3.
\end{equation}
This velocity field is smooth and fulfills $\nabla\cdot\vv{u}\equiv 0$. It was
chosen such that it vanishes at the boundaries, so that it can be retrieved
from the vorticity field $\gv{\omega}:=\nabla\times\vv{u}$ through the
Biot--Savart law without any boundary integral terms:
\begin{equation}
\vv{u}(x) = -\frac{1}{4\pi}\idx{\Omega}{}{\frac{x-y}{\vert x-y\vert^3}\times\gv{\omega}(y)}{y}.
\end{equation}
Analogous to the previous section, we define the particle approximation:
\begin{equation}
\gv{\omega}_h :=
\sum_{i=1}^{N}w_i\gv{\omega}(x_i)\delta(x-x_i).
\end{equation}
Experiments in the previous section suggested a choice of $\sigma = C\sqrt{h}$,
with $C$ between $0.25$ and $0.5$. We consequently choose $C=0.375$ and obtain
after applying the method to each component a smoothed approximation
$\gv{\omega}_\sigma$ with an anticipated convergence rate of $\bigO{\sigma^2}$
in the $L^2$-norm. In order to evaluate the Biot--Savart law for this vorticity
field, we chose the coarsest level $l$ such that the corresponding mesh width
$2^{-l}$ is smaller than $\sigma$. We then compute the orthogonal projection of
$\gv{\omega}_\sigma$ onto the \emph{standard} finite element space of piecewise
linear functions on that level. The Biot--Savart integral can then be computed
by summing over the tetrahedra, for which Suh published analytic formulas~\cite{suh2000}.
We couple these formulas with a fast multipole method~\cite{greengard1987,dehnen2002}
for the far-field evaluation. The resulting velocity field is approximated by
taking the nodal interpolation onto the standard finite element space of piecewise
\emph{quadratics} to obtain an approximate velocity field $\vv{u}_\sigma$.

Most conventional schemes apply the fast multipole method directly to the particle
field, leading to a complexity of $\bigO{N}=\bigO{h^{-d}}$, with a large hidden
constant. Note that in our case the method is applied to the coarser smoothed
approximation, leading to a complexity of only $\bigO{h^{-\frac{d}{2}}}$.

\cref{fig:velocity} shows the $L^2$-errors in the approximate smoothed vorticity
field $\gv{\omega}_\sigma$ and the velocity field $\vv{u}_\sigma$. The smoothed
vorticity field converges at a rate of $\bigO{\sigma^2}=\bigO{h}$ as expected.
With the Calder\'on--Zygmund inequality~\cref{eqn:calzyg} we obtain that the
same error bound holds for the velocity in the $W^{1,2}$-norm. As the results
indicate, in the $L^2$-norm the error seems to reduce by one power in $\sigma$
faster, resulting in a rate of $\bigO{h^{1.5}}$.
\begin{figure}
\centering
\begin{tikzpicture}
\begin{semilogyaxis}%
[
	title  = {$C=0.375$},
	xlabel = {Refinement Level $l$},
	ylabel = {$L^2$-Error},
	grid   = {major},
	xmin   = {0},
    xmax   = {8},
    ymin   = {1e-6},
	ymax   = {1},
    legend entries = {vorticity,velocity},
	legend cell align = {left},
]
\pgfplotsset{xtick={0,1,2,3,4,5,6,7,8}}
\pgfplotsset{every axis legend/.append style={
at={(.05,.05)},
anchor=south west}}
\addplot table[x=l,y=vorticity]{velocity.dat};
\addplot table[x=l,y=velocity ]{velocity.dat};
\addplot[domain=3:7]{2^(-1.00*(x+1.0)} node[above=5,pos=0.5] {$\bigO{h}$};
\addplot[domain=3:7]{2^(-1.50*(x+3.5)} node[above=5,pos=0.5] {$\bigO{h^{1.5}}$};
\end{semilogyaxis}
\end{tikzpicture}
\caption{\label{fig:velocity}$L^2$-Error of the smoothed vorticity approximation
$\gv{\omega}_\sigma$ and the resulting finite-element approximation of the
corresponding velocity $\vv{u}_\sigma$.}
\vspace{\baselineskip}
\centering
\begin{tikzpicture}
\begin{loglogaxis}%
[
	title  = {$C=0.25$, $s = 2$},
	xlabel = {Stabilization Parameter $\varepsilon$},
	ylabel = {$\cond(\sm{D^{-1}A_h})$},
	grid   = {major},
	xmin   = {1e-6},
    xmax   = {1},
    ymin   = {1e1},
	ymax   = {1e6},
    xminorticks = {false},
    legend entries = {$l=2$, $l=3$, $l=4$},
	legend cell align = {left},
]
\pgfplotsset{every axis legend/.append style={at={(.95,.95)},anchor= north east}}
\addplot table[x=eps,y=l2]{condition.dat};
\addplot table[x=eps,y=l3]{condition.dat};
\addplot table[x=eps,y=l4]{condition.dat};
\end{loglogaxis}
\end{tikzpicture}
\caption{\label{fig:cond}Condition number of the diagonally scaled system matrix $\sm{D^{-1}A_h}$.}
\end{figure}

\subsection{System Condition Number}\label{sec:conditioning}
In this section we investigate the effect of the stabilization parameter $\varepsilon$
on the condition number of the system matrix. In \cref{sec:scalar_field} we observed
instabilities on the coarse levels $l$ in the case $C=0.25, s=2$. We therefore chose 
this particular configuration for our experiments. We used the following set of
functions as our basis for the spaces $V_\sigma^P(\Omega_\sigma)$:
\begin{equation}
\mathcal{B}_\sigma^P := \biggl\lbrace
\varphi_i\biggl(\frac{x-x_i}{\sigma}\biggr)^\alpha 
\,\biggl.\biggr|\,
Q_i\in\Omega_\sigma, |\alpha|\leq P
\biggr\rbrace.
\end{equation}
We may assign a numbering $\mathcal{I} = \lbrace 1,\ldots,n\rbrace$ to this set, and
subsequently refer to its members as $\mathcal{B}_\sigma^P\ni\psi_k$, $k\in\mathcal{I}$.
We can then define the system matrix $\sm{A_h}\in\mathbb{R}^{n\times n}$ via the relation:
\begin{equation}
\sv{e_k}^\top\sm{A_h}\sv{e_l} = A_h(\psi_k,\psi_l) = a_h(\psi_k,\psi_l) + \varepsilon j(\psi_k,\psi_l),
\qquad\forall k,l\in\mathcal{I},
\end{equation}
where $\sv{e_k}\in\mathbb{R}^n$ refers to the $k$-th Cartesian basis vector, and
the approximate bilinear form $a_h$ is defined as described in the numerical setup
(\cref{sec:setup}). We are then interested in the condition number of the
diagonally scaled matrix $\sm{D^{-1}A_h}$, where $\sm{D}:=\diag\sm{A_h}$.

\Cref{fig:cond} shows the condition number of $\sm{D^{-1}A_h}$ for various
refinements levels~$l$ as a function of $\varepsilon$. In the case $l=2$ the
quadrature error is so large that the resulting matrix $\sm{A_h}$ ceased being
positive definite for $\varepsilon = 10^{-3}$, and is even singular for
$\varepsilon = 0$. This explains the large error observed in \cref{sec:scalar_field}
for this case. But even then a sufficiently large choice of $\varepsilon$ results
in a well conditioned system. For the finer refinement levels a choice of
$\varepsilon$ between $10^{-3}$ and $10^{-1}$ seems to be optimal and reduces
the matrix' condition number below 100. The effect becomes slightly less
pronounced with increasing $l$. We can thus conclude that for such a choice of
$\varepsilon$ the stabilization removes the ill-conditioning of the system,
especially in the presence of moderate quadrature errors.

\section{Conclusions and Outlook}
We have presented a new method to tackle the particle regularization problem,
based on a stabilized fictitious domain formulation with smooth shape-functions.
Our approach enjoys all the benefits of the conventional blob-methods: the
resulting smoothed approximations are $C^\infty$ functions and conserve all
moments up to order $P$. On top of that, our approach can accurately handle
general geometries. The evaluation of the smoothed approximations is cheap and
straightforward and does not require a summation over all particles as in the
case of blob functions. The fact that we can only achieve a convergence rate of
$\mathcal{O}(h^{\frac{m+1}{2}})$ as opposed to $\mathcal{O}(h^{m+1})$ might seem
disappointing, but is intrinsic to the smoothing problem at hand. This can be
illustrated in a simple one-dimensional example: given an interval of length $h$
on the real line, the $m$-node Gaussian quadrature rule will have an error bound
of $\mathcal{O}(h^{2m})$. With $m$ function values, however, we can only
construct an interpolation polynomial of degree $m-1$, having the halved error
bound $\mathcal{O}(h^m)$. \Cref{thm:optimality} shows that the smoothed
approximation is essentially just as accurate as the particle field. This also
means that it can be used to reinitialize overly distorted particle fields.
Furthermore, this means that the smoothed vorticity field has much greater
length-scales than the particle spacing.

As a consequence the velocity evaluation -- usually the most expensive part of
vortex methods -- can be drastically sped up. In our numerical experiments we
gave an example of a simple mesh-based scheme for this, which was chosen because
of its simplicity. A disadvantage of this approach is that the resulting velocity
approximation ceases being divergence-free. For future research it would be
interesting to make use of the fact that in three-dimensional space the
Biot--Savart law is of the form $\vv{u}=\nabla\times(G\star\gv{\omega})$, where
$G(x)= (4\pi|x|)^{-1}$ denotes the fundamental solution of the Laplacian. The
\enquote{curl spaces} $\nabla\times\bigl(V_\sigma^P(\Omega)\bigr)^3$ would thus
be a more natural choice for approximating the velocity, while also being
divergence-free in the strong, pointwise sense.

It is not clear whether the exact variational formulation~\cref{eqn:varprob} is
actually unstable without stabilization. A result by Reusken~\cite[Theorem~5]%
{reusken2008} indicates that a rescaling might be sufficient to achieve a stable
formulation. On the other hand, this result assumes that the bilinear form $a$
can be computed exactly. The experiments of \cref{sec:conditioning} suggest
that stabilization is especially beneficial in the presence of quadrature
errors.

Our current approach uses a uniform grid size $\sigma$ and a fixed polynomial
degree~$P$. The partition of unity approach, however, is general enough to be
extended to adaptive grids and varying polynomial degrees. For the future,
experiments with $\sigma$-, $P$-, or $\sigma P$-adaptive schemes are another
interesting field for further research.

We believe the stabilized fictitious domain approach with smooth shape functions
is not only useful for vortex particle regularization but also for other problems
which require higher degrees of smoothness, such as certain problems from linear
elasticity like the Kirchhoff--Love thin-plate theory.

The source code of the software used to obtain the results of this article can
be obtained from the authors upon request.

\section*{Acknowledgements}
We would like to thank the reviewers for their fruitful comments and suggestions.
The first author would also like to express his gratitude to Sven Gro{\ss}, Arnold
Reusken, and all the members of the DROPS team at the Lehrstuhl f\"ur Numerische
Mathematik (LNM) at RWTH Aachen University, Germany, with whom he previously
worked on two-phase flows and the ghost penalty stabilization. Without their
support and the knowledge received during that time, this work would not have
been possible. 

Last but not least, the first author receives the MEXT scholarship of the Japanese
Ministry of Education and was supported by the Keio Leading Edge Laboratory
of Science and Technology (no grant numbers allotted). Without their financial
support this research would have been impossible to conduct.

\appendix
\bibliographystyle{siamplain}
\bibliography{papers}
\end{document}